\documentclass{article}
\usepackage[T2A,T1]{fontenc}
\usepackage[utf8]{inputenc}
\usepackage[russian,english]{babel}
\usepackage{mathtools}
\usepackage{amsthm}
\usepackage{amssymb}

\newtheorem{theorem}{Theorem}
\newtheorem{proposition}{Proposition}
\newtheorem{lemma}{Lemma}
\newtheorem{corollary}{Corollary}

\DeclareMathOperator{\GL}{GL}
\DeclareMathOperator{\F}{\mathbb{F}}

\title{Random unipotent Sylow subgroups of groups of Lie type of bounded rank}
\author{Saveliy V. Skresanov}
\date{}

\begin{document}
\maketitle

\begin{abstract}
	In 2001 Liebeck and Pyber showed that a finite simple group of Lie type is a product of \( 25 \) carefully chosen
	unipotent Sylow subgroups. Later, in a series of works it was shown that \( 4 \) unipotent Sylow subgroups suffice.
	We prove that if the rank of a finite simple group of Lie type \( G \) is bounded,
	then \( G \) is a product of \( 11 \) random unipotent Sylow subgroups with probability tending to \( 1 \) as \( |G| \) tends to infinity.
	An application of the result to finite linear groups is given. The proofs do not depend on the classification of finite simple groups.
\end{abstract}

\section{Introduction}

It is often useful to express a group as a product of certain subgroups,
the \( BNB \)-decomposition of groups of Lie type being one famous example.
Motivated by a model-theoretic result of Hrushovski and Pillay~\cite[Proposition~4.3]{hrushovskiPillay}
and applications to arithmetic groups, Liebeck and Pyber proved~\cite[Theorem~B]{liebeckPyber}
that a finite simple group of Lie type is a product of \( 25 \) unipotent Sylow subgroups;
recall that a Sylow subgroup is \emph{unipotent} if it is a \( p \)-subgroup, where \( p \) is the defining characteristic of the group of Lie type in question.
Since then the problem of reducing the number of unipotent Sylow subgroups in such a decomposition received much attention.
In~\cite{babaiNP} it was announced that a finite simple group of Lie type is a product of \( 5 \) unipotent Sylow subgroups,
in~\cite{vavilovSS} it was shown that for not-twisted finite simple groups of Lie type only \( 4 \) unipotent Sylow subgroups are required,
and, finally, Garonzi, Levy, Mar\'oti, Simion~\cite{marotiUUUU} and, independently, Smolensky~\cite{smolensky} proved
that \( 4 \) unipotent Sylow subgroups are enough for all finite simple groups of Lie type.
To state the result in the form it was proved in~\cite{marotiUUUU}, recall that a finite simple group of Lie type
contains a pair of \emph{opposite} unipotent Sylow subgroups (corresponding to groups of upper triangular and lower triangular
matrices for matrix groups), see Section~\ref{rand} for more details.

\begin{proposition}[{\cite{marotiUUUU}}]\label{uuuu}
	Let \( G \) be a finite simple group of Lie type.
	Let \( U \) and \( V \) be opposite unipotent Sylow subgroups of \( G \). Then \( G = UVUV \).
\end{proposition}

In general, this result is best possible as there are finite simple groups of Lie type which are not a product of \( 3 \)
unipotent Sylow subgroups~\cite{marotiUUUU}.

In the mentioned results~\cite{liebeckPyber, babaiNP, vavilovSS, marotiUUUU, smolensky} the group \( G \) is a product of specifically chosen unipotent Sylow subgroups,
usually \( G = UVUV \cdots UV \), where \( U \) and \( V \) are opposite unipotent Sylow subgroups.
It is often the case with finite simple groups that certain existence results hold for random elements of the group,
for example, it is well-known~\cite{liebeckShalev} that every finite simple group \( G \) is generated by a pair of random elements with probability
tending to \( 1 \) as \( |G| \) tends to infinity.

It is therefore natural to ask whether one can obtain a ``randomized'' decomposition of a finite simple group of Lie type
into a product of unipotent Sylow subgroups. To this end, Pyber posed the following question (personal communication):
\smallskip

\noindent
\textbf{Question.}
\emph{Does there exist a universal constant \( C \) such that a finite simple group of Lie type \( G \)
is a product of \( C \) random unipotent Sylow subgroups with probability tending to \( 1 \) as \( |G| \) tends to infinity?}
\smallskip

Note that since all unipotent Sylow subgroups are conjugate, we can reformulate Pyber's question for a fixed unipotent Sylow subgroup as
\[ \lim_{|G| \to \infty} \frac{|\{ (g_1, \dots, g_C) \in G \times \cdots \times G \mid G = U^{g_1} \cdots U^{g_C} \}|}{|G|^C} = 1, \]
where \( G \) in the limit ranges over all finite simple groups of Lie type and \( U \) is some unipotent Sylow subgroup of~\( G \).

The main result of this paper is a positive answer to this question in the case of finite simple groups of Lie type of bounded rank.
Recall that a finite group \( G \) is called \emph{quasisimple}, if \( G \) is perfect and the quotient of \( G \) by its center \( Z(G) \) is a simple group.

\begin{theorem}\label{main}
	Let \( G \) be a finite quasisimple group, where \( G/Z(G) \) is a simple group of Lie type of bounded rank over the field of characteristic~\( p \).
	Then \( G \) is a product of \( 11 \) random Sylow \( p \)-subgroups with probability tending to \( 1 \) as \( |G| \) tends to infinity.
\end{theorem}
\begin{corollary}\label{simplecor}
	Let \( G \) be a finite simple group of Lie type of bounded rank. Then \( G \) is a product of \( 11 \) random unipotent Sylow subgroups
	with probability tending to~\( 1 \) as \( |G| \) tends to infinity.
\end{corollary}

The main idea of the argument is that although we cannot arrange our product of random Sylow subgroups to be of the form \( UVUV \cdots UV \)
for opposite unipotent Sylow subgroups \( U \) and \( V \), it turns out that two random unipotent Sylow subgroups are conjugate to a pair of opposite
subgroups almost surely (Theorem~\ref{randinter}). That can be used to show that the product of three random Sylow subgroups has large size (Lemma~\ref{threerand})
and then a criterion of Babai, Nikolov and Pyber applies (Lemma~\ref{trick}).

In~\cite[Proposition~4.3]{hrushovskiPillay} Hrushovski and Pillay proved that a finite subgroup of \( \GL_n(p) \),
which is generated by elements of order \( p \), can be written as a product of \( k \) groups of order \( p \),
where \( k \) depends only on~\( n \). In~\cite[Theorem~A]{liebeckPyber} Liebeck and Pyber provided a generalization of that result
and showed that a finite subgroup \( G \) of \( \GL_n(\F) \), where \( \F \) is a field of characteristic \( p \),
such that \( G \) is generated by elements of order \( p \), can be written as a product of \( 25 \) Sylow \( p \)-subgroups,
if \( p \) is large enough in terms of~\( n \). We improve their result and show that for any \( \epsilon > 0 \) the group \( G \) is a product
of \( 11 \) random Sylow \( p \)-subgroups with probability at least \( 1 - \epsilon \), if \( p \) is large enough in terms of \( n \) and \( \epsilon \).

\begin{theorem}\label{lpmain}
	For every \( \epsilon > 0 \) and \( n \geq 1 \) there exists an integer \( f_\epsilon(n) \) such that the following holds.
	Let \( p \) be a prime with \( p \geq f_\epsilon(n) \), and let \( \F \)
	be a field of characteristic~\( p \). If \( G \) is a finite subgroup of \( \GL_n(\F) \) generated by elements of order \( p \),
	then \( G \) is a product of \( 11 \) random Sylow \( p \)-subgroups with probability at least \( 1 - \epsilon \).
\end{theorem}

It is an interesting question whether the result of Hrushovski and Pillay can also be proved for random subgroups of order \( p \),
but unfortunately Theorem~\ref{lpmain} does not immediately imply that.

To put our results into a more general perspective, recall the Liebeck--Nikolov--Shalev conjecture,
which was recently proved by Gill, Lifshitz, Pyber and Szab\'o~\cite{glps, lifshitz}:
there exists a universal constant \( C > 0 \) such that for any nonabelian finite simple group \( G \) and a subset \( A \subseteq G \)
with \( |A| \geq 2 \) we have \( G = A^{g_1} \cdots A^{g_k} \) for some elements \( g_1, \dots, g_k \in G \), where \( k \leq C \cdot \log |G| / \log |A| \).
A natural extension of our work would be to prove a randomized analogue of the Liebeck--Nikolov--Shalev conjecture where elements \( g_1, \dots, g_k \)
are chosen randomly, at least for groups of Lie type of bounded rank or for sufficiently large subsets. Note that such a result
would imply the answer to Pyber's question about Sylow subgroups, since for a unipotent Sylow subgroup \( U \) the fraction \( \log |G| / \log |U| \)
is bounded by a universal constant independent of the rank.

The structure of the paper is as follows. In Section~\ref{rand} we provide some preliminary properties of groups of Lie type
and prove that a pair of random unipotent Sylow subgroups is conjugate to a pair of opposite subgroups (Theorem~\ref{randinter}).
In Section~\ref{tdecomp} we establish a certain decomposition of a finite simple group of Lie type, Section~\ref{mainsec} is devoted
to the proof of the main result (Theorem~\ref{main}) and in Section~\ref{appsec} we prove Theorem~\ref{lpmain} about subgroups of linear groups.

\section{Random Sylow subgroups}\label{rand}

We start with some preliminaries. Given two real-valued functions \( f(q) \) and \( g(q) \) we say that \( f \)
and \( g \) are \emph{asymptotically equivalent}, and write \( f \sim g \), if \( \lim_{q \to \infty} f(q)/g(q) = 1 \).
Given some family of finite groups, we will say that an event happens \emph{almost surely}, if the probability of the
event happening tends to \( 1 \) as the size of the group from the family tends to infinity.

Recall that there is a uniform distribution on all elements of a finite group, and on all Sylow \( p \)-subgroups for a fixed prime~\( p \).
Since all Sylow \( p \)-subgroups of a group \( G \) are conjugate, choosing a Sylow \( p \)-subgroup uniformly at random is the same as choosing
an element \( g \in G \) uniformly at random and computing \( U^g \) for some fixed Sylow \( p \)-subgroup \( U \) of~\( G \).
For brevity, we will sometimes write ``random element (subgroup)'' instead of ``uniformly random element (subgroup)''.

We mostly follow~\cite{carter} for our treatment of finite simple groups of Lie type.
Let \( G \) be a finite simple group of Lie type over the field of characteristic~\( p \).
Let \( U \) be a Sylow \( p \)-subgroup of \( G \), and let \( B = UH \) be a Borel subgroup.
Recall that \( G \) is a group with a \( (B, N) \)-pair, and \( W = N/H \) is the Weyl group.
Given some \( w \in W \), one can choose a representative \( w' \in N \) as \( w = Hw' \).
Since \( U \) and \( B \) are normalized by \( H \), we can correctly write \( U^w = U^{w'} \) and \( B^w = B^{w'} \).

Let \( w_0 \) be the long element of the Weyl group \( W \) of \( G \). Then \( w_0 \) is an involution,
\( U \cap U^{w_0} = 1 \) and \( B \cap B^{w_0} = H \), see~\cite[Theorem~2.3.8]{cfsg3}.
The Sylow \( p \)-subgroups \( U \) and \( U^{w_0} \) are called opposite.

If \( G \) is a Chevalley group of rank \( l \) defined over the field of order \( q \),
then by~\cite[Section~8.6]{carter}, we have \( |U| = q^M \), where \( M \) is the number of positive roots,
and \( |H| = \frac{1}{d}(q-1)^l \), where \( d \) is defined in~\cite[Section~8.6]{carter}. By~\cite[Theorem~9.4.10]{carter}, we have
\[ |G| = \frac{1}{d}q^M (q^{d_1}-1) \cdots (q^{d_l}-1), \]
where \( d_1, \dots, d_l \) are the degrees of the basic polynomial invariants of the Weyl group.
By~\cite[Theorem~9.3.4~(ii)]{carter}, we have \( d_1 + \dots + d_l = M+l \).

Note that
\begin{align*}
	|G| &= \frac{1}{d} q^{2M+l}(1 - 1/q^{d_1}) \cdots (1 - 1/q^{d_l}) \sim \frac{1}{d} q^{2M+l},\\
	|H| &= \frac{1}{d} q^l (1-1/q)^l \sim \frac{1}{d}q^l,\\
	|B| &= |U|\cdot |H| \sim \frac{1}{d} q^{M+l}.
\end{align*}

If \( G \) is a twisted group of rank \( l \), then by~\cite[Section~14.1]{carter}, we have \( |U| = q^M \)
and \( |H| = \frac{1}{d}(q - \eta_1)(q - \eta_2) \cdots (q - \eta_l) \), where \( q \), \( M \), \( d \) and \( \eta_i \) are defined in~\cite[Section~14.1]{carter}.
Note that for \( G = {}^2B_2(2^{2m+1}), {}^2G_2(3^{2m+1}), {}^2F_4(2^{2m+1}) \) we define \( q = p^{m+1/2} \) for \( p = 2, 3 \),
following~\cite[Section~14.1]{carter}, although \( q \) is an irrational number in these cases.

By~\cite[Theorem~14.3.1]{carter}, we have
\[ |G| = \frac{1}{d} q^M (q^{d_1} - \epsilon_1) \cdots (q^{d_l} - \epsilon_l), \]
where \( d_i \) and \( \epsilon_i \) are defined in~\cite[Section~14.2]{carter}. Again, by~\cite[Theorem~9.3.4~(ii)]{carter} we have \( d_1 + \dots + d_l = M+l \).

By~\cite[Section~14.3]{carter}, \( \epsilon_i = \pm 1 \) and \( \eta_i = \pm 1 \), and since \( l \) is bounded, we have
\begin{align*}
	|G| &= \frac{1}{d} q^{2M+l}(1 - \epsilon_1/q^{d_1}) \cdots (1 - \epsilon_l/q^{d_l}) \sim \frac{1}{d} q^{2M+l},\\
	|H| &= \frac{1}{d} q^l (1-\eta_1/q) \cdots (1 - \eta_l/q) \sim \frac{1}{d}q^l,\\
	|B| &= |U|\cdot |H| \sim \frac{1}{d} q^{M+l}.
\end{align*}

Observe that with our notation the asymptotic formulae for \( |G|, |H| \) and \( |B| \) are the same for Chevalley groups and twisted groups.

\begin{theorem}\label{randinter}
	Let \( G \) be a finite simple group of Lie type.
	If the rank of \( G \) is bounded, then two randomly chosen unipotent Sylow subgroups of \( G \)
	are conjugate to a pair of opposite subgroups with probability tending to \( 1 \) as \( |G| \) goes to infinity.
\end{theorem}
\begin{proof}
	Let \( q \), \( l \), \( M \), \( N \), \( B \), \( H \) and \( U \) be as we described above,
	and let us identify \( w_0 \in W \) with one of its representatives in \( N \).
	Since all unipotent Sylow subgroups of \( G \) are conjugate, it is enough to prove that a randomly chosen \( g \in G \)
	lies in \( Bw_0B \) with probability tending to~\( 1 \) when \( q \) goes to infinity.
	Indeed, for such \( g \) we have \( g = b_1w_0b_2 \), \( b_1, b_2 \in B \), hence \( U^g = U^{b_1w_0b_2} = U^{w_0b_2} \).
	On the other hand, \( U = U^{b_2} \), so \( U \) and \( U^g \) are conjugate to \( U \) and \( U^{w_0} \).

	Now, to finish the proof, recall that \( B \cap B^{w_0} = H \), so
	\[ \frac{|Bw_0B|}{|G|} = \frac{|B|^2}{|G| \cdot |H|} \sim \frac{\frac{1}{d^2} q^{2M+2l}}{\frac{1}{d} q^{2M+l} \cdot \frac{1}{d} q^l} = 1. \]
	It follows that when \( l \) is bounded, \( |Bw_0B|/|G| \) tends to \( 1 \) as \( q \) tends to infinity, hence the claim is proved.
\end{proof}

As a simple corollary, under the assumptions of Theorem~\ref{randinter}, two random unipotent Sylow subgroups intersect trivially almost surely.

Given a finite simple group of Lie type \( G \) with a Borel subgroup \( B \), the distribution of \( |BgB| \) for random \( g \in G \)
is related to the so-called Mallows measure~\cite{diaconisSimper}. It follows from~\cite[Theorem~3.5]{diaconisSimper} that for \( G = \mathrm{PSL}_n(q) \)
where \( q \) is fixed while \( n \) tends to infinity (a situation opposite to what we study in this paper), two random unipotent Sylow subgroups
intersect nontrivially most of the time.

\section{Toffoli-like decomposition}\label{tdecomp}

In~\cite[Theorem~2.7]{toffoli}, Toffoli showed that any real matrix with determinant \( 1 \) up to some row permutation and sign changes
can be decomposed as a product of an upper unitriangular, lower unitriangular and again an upper unitriangular matrix.
The following result is a generalization of this decomposition for finite simple groups of Lie type and might be already known,
though the author has not been able to find it in the literature.

\begin{proposition}\label{toffoliGen}
	Let \( G \) be a finite simple group of Lie type with a \( (B, N) \)-pair.
	Let \( U \) and \( V \) be opposite unipotent Sylow subgroups of \( G \), and let \( W \) be its Weyl group.
	If \( n_w \), \( w \in W \) is a complete system of representatives of elements of \( W \) in \( N \), then
	\[ G = \bigcup_{w \in W} UVUn_w, \]
	in particular, \( |UVU| \geq |G| / |W| \).
\end{proposition}
\begin{proof}
	We use the notation of Section~\ref{rand}.
	For brevity, let \( n = n_{w_0} \), so \( V = U^n \). Let \( h \in H \) be some arbitrary element, and recall that \( h \) normalizes \( U \) and \( V \).
	By Proposition~\ref{uuuu}, we have \( hn \in UVUV \).
	Hence there exists \( u \in U \) such that \( hn \cdot u^{hn} = uhn \in VUV \). Since \( u^{hn} \in V \), we obtain \( hn \in VUV \),
	and as \( V^{hn} = U \), we have \( hn \in UVU \).

	Since \( h \in H \) was arbitrary, we have \( Hn \subseteq UVU \) and thus
	\[ Bn = UHn \subseteq U \cdot UVU = UVU. \]
	By multiplying by \( U \) from the right, we get \( BVn = BnU \subseteq UVU \).
	Now, when \( w \) runs through all elements of \( W \), the product \( nn_w \)
	gives a complete system of representatives of elements from \( W \) in \( N \), hence
	\[ \bigcup_{w \in W} UVUn_w \supseteq \bigcup_{w \in W} BVnn_w = \bigcup_{w\in W} BVn_w. \]
	If \( G \) is a Chevalley group, then by~\cite[Theorem~8.4.3]{carter},
	\[ G = \bigcup_{w \in W} Bn_wU_w^-, \]
	where \( U_w^- \) is a certain product of root subgroups as defined in~\cite[Section~8.4]{carter}.
	The proof of~\cite[Theorem~8.4.3]{carter} shows that \( n_wU_w^- n_w^{-1} \leq V \), hence \( Bn_wU_w^- \subseteq BVn_w \).
	Therefore \( G = \cup_{w \in W} UVUn_w \) in this case.

	If \( G \) is a twisted group, then let \( \overline{G} \) be the corresponding Chevalley group with an automorphism \( \sigma \),
	see~\cite[Section~13.4]{carter}. Let \( \overline{U} \) and \( \overline{V} \) be opposite unipotent Sylow subgroups of \( \overline{G} \)
	such that \( U = \{ x \in \overline{U} \mid x^\sigma = x \} \) and \( {V = \{ x \in \overline{V} \mid x^\sigma = x \}} \). 
	By~\cite[Proposition~13.5.3]{carter}, any element \( g \in G \) decomposes as \( g = bn_wu \) for some
	\( b \in B \), \( w \in W \) and \( u \in U' \), where \( U' \) is the set of \( \sigma \)-invariant elements of
	\( \overline{U}_w^- \leq \overline{U} \). We want to prove that \( g \in BVn_w \); note that since \( H \leq B \), the choice
	of the representative \( n_w \) from \( n_wH \) does not matter.
	It is shown in the proof of~\cite[Proposition~13.5.3]{carter} that \( n_w \) can be chosen to be \( \sigma \)-invariant,
	so without loss of generality we may assume that \( n_w^\sigma = n_w \).

	As noted earlier, in the Chevalley group \( \overline{G} \) we have \( n_w \overline{U}_w^- n_w^{-1} \leq \overline{V} \).
	Hence \( n_w U' n_w^{-1} \leq \overline{V} \) and as the left hand side is \( \sigma \)-invariant, we have \( n_w U' n_w^{-1} \leq V \).
	Therefore \( g \) lies in \( BVn_w \) and since \( g \in G \) was arbitrary, \( G = \cup_{w \in W} BVn_w \).
	As in the case of Chevalley groups, this implies \( G = \cup_{w \in W} UVUn_w \), as wanted.
\end{proof}

\section{Proof of Theorem~\ref{main}}\label{mainsec}

The following result was used by Babai, Nikolov and Pyber to prove that any finite simple group of Lie type is a product of five unipotent Sylow subgroups.

\begin{lemma}[{\cite[Corollary~2.6]{babaiNP}}]\label{trick}
	Let \( G \) be a finite group, and let \( k \) be the minimum degree of a nontrivial complex representation of \( G \).
	If \( A_1, \dots, A_t \subseteq G \) are nonempty, and
	\[ \prod_{i=1}^t |A_i| \geq \frac{|G|^t}{k^{t-2}}, \]
	then \( \prod_{i=1}^t A_i = G \).
\end{lemma}

Recall the parameters \( q \), \( l \), \( M \) and \( d \) we introduced for finite simple groups of Lie type in Section~\ref{rand},
so if \( G \) is such a group with a unipotent Sylow subgroup \( U \), then \( |G| \sim 1/d \cdot q^{2M+l} \) and \( |U| = q^M \).
In the first three columns of Table~\ref{tab} below we list the values of these parameters for groups of Lie type, see~\cite[Sections~3.6 and~8.6]{carter}
for Chevalley groups and~\cite[Section~14.3]{carter} for twisted groups.
Let \( e(G) \) denote the minimum degree of a nontrivial complex representation of \( G \).
Landazuri and Seitz~\cite{landazuriSeitz} obtained lower bounds on \( e(G) \) and we use corrected lower bounds
as listed in~\cite[Table~5.3.A]{kleidman}. In the last column of the table we list asymptotic values of \( e(G) \),
in other words, the terms of highest order. For example, for \( G = D_l(q) \simeq \mathrm{P\Omega}_{2l}^+(q) \)
we have \( e(G) \geq (q^{l-1} - 1)(q^{l-2} + 1) \) when \( q \neq 2, 3, 5 \) and \( e(G) \geq q^{l-2}(q^{l-1}-1) \) when \( q = 2, 3, 5 \).
In both of these cases, the lower bound is asymptotically equivalent to \( q^{2l-3} \) as \( q \to \infty \), and we list this formula in the relevant cell of the table.

\begin{table}[h]
\begin{tabular}{c | c c c}
	Group & \( M \) & \( d \) & Asymptotic of \( e(G) \)\\
	\hline
	\( A_l(q) \) & \( l(l+1)/2 \) & \( (l+1, q-1) \) &
		\begin{tabular}{l}\( \frac{1}{(2, q-1)} \cdot q \), if \( l = 1 \), \\ \( q^l \), if \( l \geq 2 \) \end{tabular}\\
	\( {}^2A_l(q^2) \), \( l \geq 2 \) & \( l(l+1)/2 \) & \( (l+1, q+1) \) & \( q^l \)\\
	\( B_l(q) \), \( l \geq 3 \), \( q \) odd & \( l^2 \) & \( (2, q-1) \) & \( q^{2l-2} \)\\
	\( C_l(q) \), \( l \geq 2 \) & \( l^2 \) & \( (2, q-1) \) &
		\begin{tabular}{l} \( \frac{1}{2} q^l \), if \( q \) is odd,\\ \( \frac{1}{2} q^{2l-1} \), if \( q \) is even \end{tabular}\\
	\( D_l(q) \), \( l \geq 4 \) & \( l(l-1) \) & \( (4, q^l - 1) \) & \( q^{2l-3} \)\\
	\( {}^2D_l(q^2) \), \( l \geq 4 \) & \( l(l-1) \) & \( (4, q^l + 1) \) & \( q^{2l-3} \)\\
	\( G_2(q) \) & \( 6 \) & \( 1 \) & \( q^3 \)\\
	\( F_4(q) \) & \( 24 \) & \( 1 \) & \begin{tabular}{l} \( q^8 \), if \( q \) is odd,\\ \( \frac{1}{2} q^{11} \), if \( q \) is even \end{tabular}\\
	\( E_6(q) \) & \( 36 \) & \( (3, q-1) \) & \( q^{11} \)\\
	\( E_7(q) \) & \( 63 \) & \( (2, q-1) \) & \( q^{17} \)\\
	\( E_8(q) \) & \( 120 \) & \( 1 \) & \( q^{29} \)\\
	\( {}^2E_6(q^2) \) & \( 36 \) & \( (3, q+1) \) & \( q^{11} \)\\
	\( {}^3D_4(q^3) \) & \( 12 \) & \( 1 \) & \( q^5 \)\\
	\( {}^2B_2(q^2) \) & \( 4 \) & \( 1 \) & \( \frac{1}{\sqrt{2}} q^3 \)\\
	\( {}^2G_2(q^2) \) & \( 6 \) & \( 1 \) & \( q^4 \)\\
	\( {}^2F_4(q^2) \) & \( 24 \) & \( 1 \) & \( \frac{1}{\sqrt{2}} q^{11} \)
\end{tabular}
\caption{Parameters of finite simple groups of Lie type}\label{tab}
\end{table}

In some cases the bound on \( e(G) \) in Table~\ref{tab} together with Lemma~\ref{trick} are enough to prove Theorem~\ref{main}.
For instance, consider the case when \( G = G_2(q) \). Let \( g_1, \dots, g_{t+1} \in G \), \( t \geq 1 \),
be some elements chosen uniformly independently at random; we will specify the value of \( t \) later.
We want to apply Lemma~\ref{trick} to sets \( A_i = U^{g_i}U^{g_{i+1}} \), \( i = 1, \dots, t \).
Indeed, if the assumptions of Lemma~\ref{trick} will hold almost surely, then
\[ G = A_1 \cdots A_t = U^{g_1} U^{g_2} \cdot U^{g_2} U^{g_3} \cdots U^{g_t} U^{g_{t+1}} = U^{g_1} U^{g_2} U^{g_3} \cdots U^{g_{t+1}} \]
and hence \( G \) will be a product of \( t+1 \) random unipotent Sylow subgroups.

To check the prerequisites of the lemma, recall that for \( G_2(q) \) we have \( {l = 2} \), \( M = 6 \), \( d = 1 \) and \( e(L) > q^3/2 \) for \( q \) large enough.
Since \( g_i \) are chosen independently at random, Theorem~\ref{randinter} implies that \( U^{g_i} \cap U^{g_{i+1}} = 1 \) almost surely,
hence \( |A_i| = q^{2M} \) almost surely. By Lemma~\ref{trick}, it suffices to check that
\[ |A_1|\cdots |A_t| \geq \frac{|G|^t}{e(G)^{t-2}} \]
for some constant \( t \). Consider the right hand side:
\[ \frac{|G|^t}{e(G)^{t-2}} \sim \frac{q^{(2M+l)t}}{e(G)^{t-2}} < \frac{q^{14t}}{(q^3/2)^{t-2}} = 2^{t-2} \cdot q^{14t - 3(t-2)} = 2^{t-2} \cdot q^{11t + 6}. \]
The left hand side is \( |A_1| \cdots |A_t| = q^{2M \cdot t} = q^{12t} \). For \( t = 7 \) we have \( 12t > 11t+6 \), hence the left hand side grows asymptotically faster
and we have the desired inequality for \( q \) large enough. Therefore \( G = G_2(q) \) is a product of \( 8 \) random unipotent Sylow subgroups.

Unfortunately, this type of argument does not work for all groups, in particular, it fails for \( G = A_l(q) \) for any choice of \( t \).
Indeed, observe that \( M = l(l+1)/2 \), \( d = (l+1, q-1) \) and \( e(G) > q^l/3 \) for \( q \) large enough.
Again, we take \( g_1, \dots, g_{t+1} \in G \) uniformly independently at random, and consider \( A_i = U^{g_i} U^{g_{i+1}} \), \( i = 1, \dots, t \).
Repeating the same argument as in the case of \( G_2(q) \), it suffices to check that
\[ q^{2M\cdot t} \geq \frac{|G|^t}{e(G)^{t-2}}. \]
We have \( |G| \sim q^{2M+l}/d \), \( d \geq 1 \), and \( e(G) > q^l/3 \), hence for the right hand side we have
\[ \frac{|G|^t}{e(G)^{t-2}} \sim \frac{(q^{2M+l}/d)^t}{e(G)^{t-2}} < \frac{q^{(2M+l)t}}{q^{l(t-2)}/3^{t-2}} = 3^{t-2} \cdot q^{2M\cdot t + 2l}. \]
This bound is asymptotically larger than the one on the left hand side (since \( 2M\cdot t + 2l > 2M \cdot t \)), so our
attempt to apply Lemma~\ref{trick} fails for any \( t \). The main problem is that the product of two random unipotent Sylow subgroups
is too small for this argument to work, and in order to circumvent this issue, we need to obtain a good lower bound on the
product of \emph{three} random unipotent Sylow subgroups.

The following observation is proved by direct inspection of the last column of Table~\ref{tab}.

\begin{lemma}\label{lowb}
	For \( q \) large enough we have \( e(G) > q^l/3 \).
\end{lemma}

We require a basic tool from additive combinatorics.

\begin{lemma}[{Ruzsa triangle inequality~\cite{ruzsa}}]\label{ruzsa}
	Let \( A, B, C \) be finite subsets of some group. Then \( |AB|\cdot |C| \leq |AC| \cdot |C^{-1}B| \).
\end{lemma}
\begin{corollary}\label{growth}
	Let \( A, B \) be finite subsets of some group. Then \( \sqrt{|AA^{-1}|\cdot |B|} \leq |AB| \).
\end{corollary}
\begin{proof}
	Apply Ruzsa triangle inequality to subsets \( A, A^{-1} \) and \( B \), and notice that \( |AB| = |B^{-1}A^{-1}| \).
\end{proof}

The product of three random unipotent Sylow subgroups is large almost surely.

\begin{lemma}\label{threerand}
	Let \( G \) be a finite simple group of Lie type of bounded rank, and let \( q, M, l, d, U \) and \( W \) be defined as in Section~\ref{rand}.
	If \( g_1, g_2, g_3 \in G \) are chosen uniformly independently at random, then
	\[ |U^{g_1} U^{g_2} U^{g_3}| \geq \frac{q^{2M + l/2}}{\sqrt{d \cdot |W|}} \]
	with probability tending to~\( 1 \) as \( |G| \) tends to infinity.
\end{lemma}
\begin{proof}
	By Theorem~\ref{randinter}, \( U^{g_i} \) and \( U^{g_{i+1}} \) are conjugate to a pair of opposite unipotent Sylow subgroups for \( i = 1, 2 \)
	with probability tending to \( 1 \) as \( |G| \to \infty \). In particular, \( |U^{g_i}U^{g_{i+1}}| = q^{2M} \), \( i = 1, 2 \).

	Since \( U^{g_1} \) and \( U^{g_2} \) are conjugate to \( U \) and \( V \), by Proposition~\ref{toffoliGen} we have \( |U^{g_1} U^{g_2} U^{g_1}| \geq |G|/|W| \).
	Now we apply Corollary~\ref{growth} to sets \( A = U^{g_1} U^{g_2} \) and \( B = U^{g_2} U^{g_3} \):
	\[ |U^{g_1} U^{g_2} U^{g_3}| = |U^{g_1} U^{g_2} \cdot U^{g_2} U^{g_3}| \geq \sqrt{|U^{g_1} U^{g_2} U^{g_1}| \cdot |U^{g_2} U^{g_3}| }. \]
	To prove the claim we plug in the lower bound for \( |U^{g_1} U^{g_2} U^{g_1}| \) and use the asymptotic formula for \( |G| \):
	\[ |U^{g_1} U^{g_2} U^{g_3}| \geq
	\sqrt{\frac{|G|}{|W|} \cdot q^{2M}} \sim \sqrt{\frac{q^{2M+l}}{d \cdot |W|} \cdot q^{2M}} = \frac{q^{2M + l/2}}{\sqrt{d \cdot |W|}}. \]
\end{proof}

Now we are ready to prove the main result.
\smallskip

\noindent\emph{Proof of Theorem~\ref{main}.}
Let \( \overline{\phantom{a}} : G \to G/Z(G) \) denote the natural homomorphism, and let \( U \) be a Sylow \( p \)-subgroup of \( G \).
Clearly, \( \overline{U} \) is a unipotent Sylow subgroup of \( \overline{G} \). Define \( q, M, l, d \) and \( W \) for \( \overline{G} \) as in Section~\ref{rand}.
It follows from the structure of Schur multipliers of groups of Lie type~\cite[Theorem~5.1.4]{kleidman}, that \( |Z(G)| \leq \max \{ 48,\, d \} \).
In particular, \( |Z(G)| \leq 48l \).

Let \( g_1, \dots, g_{11} \in G \) be chosen uniformly independently at random.
Set \( A_i = U^{g_{2i-1}} U^{g_{2i}} U^{g_{2i+1}} \), \( i = 1, \dots, 5 \).
Notice that images \( \overline{g_1}, \dots, \overline{g_{11}} \) of \( g_1, \dots, g_{11} \) are also uniformly random in \( \overline{G} \),
hence we may assume that with probability tending to \( 1 \) as \( |G| \to \infty \),
the conclusion of Lemma~\ref{threerand} holds for sets
\[ \overline{A_i} = \overline{U^{g_{2i-1}} U^{g_{2i}} U^{g_{2i+1}}} = \overline{U}^{\overline{g_{2i-1}}} \overline{U}^{\overline{g_{2i}}} \overline{U}^{\overline{g_{2i+1}}},\,\, i = 1, \dots, 5. \]
Therefore
\[ |A_i| \geq |\overline{A_i}| \geq \frac{q^{2M+l/2}}{\sqrt{d \cdot |W|}},\,\, i = 1, \dots, 5. \]

Now we want to apply Lemma~\ref{trick} to \( t = 5 \) subsets \( A_i \), \( i = 1, \dots, 5 \).
It suffices to check the inequality
\[ |A_1| \cdots |A_t| \geq \frac{|G|^t}{e(G)^{t-2}}. \tag{$\star$} \]
The left hand side has an asymptotic lower bound of
\[ |A_1| \cdots |A_t| \geq \frac{q^{(2M + l/2)t}}{(d \cdot |W|)^{t/2}}. \]
By Table~\ref{tab}, we have a crude bound \( d \leq 4l \), hence
\[ |A_1| \cdots |A_t| \geq \frac{q^{(2M + l/2)t}}{(4l \cdot |W|)^{t/2}}. \]
By Lemma~\ref{lowb}, \( e(\overline{G}) > q^l/3 \), and by~\cite[Corollary~5.3.3]{kleidman}, \( e(G) \geq e(\overline{G}) \),
hence \( e(G) > q^l/3 \). Since \( |G| = |Z(G)| \cdot |\overline{G}| \leq 48l \cdot |\overline{G}| \) and \( |\overline{G}| \sim q^{2M+l}/d \),
the right hand side of~\( (\star) \) has an asymptotic upper bound of
\[ \frac{|G|^t}{e(G)^{t-2}} \sim \frac{|Z(G)|^t q^{(2M+l)t}/d^t}{e(G)^{t-2}} \leq
   \frac{(48l)^t \cdot q^{(2M+l)t}}{(q^l/3)^{t-2}} = 3^{t-2} (48l)^t \cdot q^{(2M+l)t - l(t-2)}. \]
To prove~\( (\star) \) for large enough \( q \), it suffices to show that the left hand side of~\( (\star) \) grows faster than the right hand side.
Since the rank \( l \) is bounded, the Weyl group has bounded order, hence \( (4l \cdot |W|)^{t/2} \) is bounded.
It therefore suffices to compare the powers of \( q \) and check that
\[ (2M + l/2)t > (2M+l)t - l(t-2). \]
This inequality is equivalent to \( t > 4 \), so it holds for our choice of \( t = 5 \).

Finally, since~\( (\star) \) holds, Lemma~\ref{trick} implies that \( G = A_1 \cdots A_5 = U^{g_1} \cdots U^{g_{11}} \). The theorem is proved. \qed

\section{Proof of Theorem~\ref{lpmain}}\label{appsec}

For a prime \( p \), let \( O_p(G) \) denote the largest normal \( p \)-subgroup of \( G \).
We need the following corollary from the Larsen--Pink theorem~\cite{larsenPink}.
The proof does not require the classification of finite simple groups (CFSG).

\begin{lemma}\label{lp}
	Let \( G \) be a finite subgroup of \( \GL_n(\F) \), where \( \F \) is a field of characteristic \( p > 2 \).
	Then \( G \) has a normal subgroup \( N \geq O_p(G) \) of index bounded in terms of \( n \) only,
	such that \( N/O_p(G) \) is a central product of an abelian group \( A \) and quasisimple groups \( Q_1, \dots, Q_k \).
	Moreover, \( |A| \) is not divisible by \( p \), \( Q_i/Z(Q_i) \) is a group of Lie type in characteristic \( p \) for all \( i = 1, \dots, k \),
	and we can bound \( k \) and the ranks of \( Q_i/Z(Q_i) \) in terms of \( n \) only.
\end{lemma}
\begin{proof}
	The statement of the lemma is almost identical to~\cite[Corollary~3.1]{liebeckPyber}, which was proved as a corollary of the Larsen--Pink theorem,
	with the only exception that there was no bound on \( k \) and the ranks of \( Q_i/Z(Q_i) \), \( i = 1, \dots, k \).
	To finish the proof we need to bound these parameters in terms of~\( n \).

	Recall that a section of a group is a homomorphic image of a subgroup. Fix some \( i \in \{ 1, \dots, k \} \),
	and let \( l \) be the rank of \( Q_i/Z(Q_i) \). It follows from the structure of Weyl groups of finite simple groups of Lie type,
	that \( Q_i/Z(Q_i) \) contains a section isomorphic to the alternating group of degree at least \( c \cdot l \) for some universal constant \( c > 0 \)
	(see~\cite[Table~1]{revin} for explicit bounds).
	This alternating group is also a section of \( \GL_n(\F) \), hence by~\cite[Theorem~5.7A]{dixonMortimer},
	\( n \geq (2c \cdot l - 4)/3 \) and \( l \) is bounded in terms of~\( n \).

	To bound \( k \), choose elements \( x_i \in Q_i \), \( i = 1, \dots, k \),
	such that the image of \( x_i \) in \( Q_i/Z(Q_i) \), \( i = 1, \dots, k \), is an involution.
	The group generated by \( x_1, \dots, x_k \) is an abelian subgroup of \( N/O_p(G) \), let \( K \) be its Sylow \( 2 \)-subgroup.
	Since \( p \neq 2 \), there exists a subgroup \( L \) of \( N \) with \( L \cap O_p(G) = 1 \) and \( LO_p(G)/O_p(G) = K \).
	Clearly \( L \) is an abelian \( 2 \)-subgroup of \( \GL_n(\F) \) of rank (as an abelian group) at least~\( k \).
	By Maschke's theorem, the natural representation of \( L \) on \( \F^n \) decomposes into a sum of, say, \( m \) irreducible representations,
	and a quotient of \( L \) by the kernel of each irreducible representation is cyclic~\cite[Theorem~3.2.3]{gorenstein}.
	Therefore \( L \) is a subdirect product of \( m \) cyclic groups, hence its rank is at most \( m \leq n \).
	Thus \( k \leq n \) and we are done.
\end{proof}

We note that the statement above is true for \( p = 2 \), but would require a bit more technicalities to prove.
Explicit and more precise bounds on the parameters can be obtained through CFSG, see, for example, Weisfeiler's result~\cite{weisfeilerJordan}.
We also note that an explicit CFSG-free version of the Larsen--Pink theorem was obtained in~\cite{donaLP}.
\smallskip

\noindent
\emph{Proof of Theorem~\ref{lpmain}.}
We follow the proof of~\cite[Theorem~A]{liebeckPyber}. Fix \( n \geq 1 \) and \( \epsilon > 0 \).
Let \( G \) be a finite subgroup of \( \GL_n(\F) \) generated by elements of order~\( p \). 
Since we may assume \( p > 2 \), Lemma~\ref{lp} applies, and in the notation of that lemma, \( G \) has a normal subgroup \( N \) of index bounded in terms of~\( n \).
By choosing \( f_\epsilon(n) \) large enough, we may assume that \( p \) is larger than that index, therefore all elements of order \( p \)
lie in \( N \) and hence \( G = N \). Every quotient of \( G \) is also generated by elements of order \( p \), hence \( A = 1 \).
So we may assume that \( G/O_p(G) \) is a central product of quasisimple groups \( Q_1, \dots, Q_k \), where \( Q_i/Z(Q_i) \) is a quasisimple
group of Lie type in characteristic \( p \) for all \( i = 1, \dots, k \), while \( k \) and the ranks of \( Q_i/Z(Q_i) \) are bounded in terms of \( n \) only.

Since all Sylow \( p \)-subgroups of \( G \) contain \( O_p(G) \) and are in one to one correspondence with the Sylow \( p \)-subgroups
of \( G/O_p(G) \), we may assume that \( O_p(G) = 1 \). Now, set \( G^* = Q_1 \times \cdots \times Q_k \) and \( Z \leq Z(G^*) \) such that \( G \simeq G^* / Z \).
Let \( U_i \) be a Sylow \( p \)-subgroup of \( Q_i \), \( i = 1, \dots, k \), so \( U^* = U_1 \times \cdots \times U_k \)
is a Sylow \( p \)-subgroup of \( G^* \). Clearly, \( U = U^*Z/Z \) can be identified with a Sylow \( p \)-subgroup of~\( G \).

Let \( L \) be the number of tuples \( (g_1, \dots, g_{11}) \in G \times \cdots \times G \) such that \( G = U^{g_1} \cdots U^{g_{11}} \),
and let \( L^* \) be the number of tuples \( (h_1, \dots, h_{11}) \in G^* \times \cdots \times G^* \) such that \( G^* = (U^*)^{h_1} \cdots (U^*)^{h_{11}} \).
Then \( L/|G|^{11} \) is the probability that the product of \( 11 \) random Sylow \( p \)-subgroups equals \( G \),
and \( L^*/|G^*|^{11} \) is the similar probability but for \( G^* \).

Since \( k \) and the ranks of \( Q_i/Z(Q_i) \), \( i = 1, \dots, k \), are bounded, Theorem~\ref{main} applied to \( Q_i \), \( i = 1, \dots, k \),
implies that for \( f_\epsilon(n) \leq p \) large enough, we have \( L^*/|G^*|^{11} \geq 1-\epsilon \).
If \( g_1, \dots, g_{11} \in G \) have preimages \( h_1, \dots, h_{11} \in G^* \),
then clearly equality \( G^* = (U^*)^{h_1} \cdots (U^*)^{h_{11}} \) implies \( G = U^{g_1} \cdots U^{g_{11}} \). Therefore \( L \geq L^* / |Z|^{11} \) and
\[ \frac{|L|}{|G|^{11}} \geq \frac{L^* / |Z|^{11}}{|G^*|^{11} / |Z|^{11}} = \frac{|L^*|}{|G^*|^{11}} \geq 1-\epsilon. \]
The theorem is proved. \qed

\section{Acknowledgements}

The author thanks A.~Mar\'oti and L.~Pyber for fruitful discussions which improved the text considerably.

The research was carried out within the framework of the Sobolev Institute of Mathematics state contract (project FWNF-2022-0002).

\bigskip

\noindent
\emph{Saveliy V. Skresanov}

\noindent
\emph{Sobolev Institute of Mathematics,\\ 4 Acad. Koptyug avenue, 630090 Novosibirsk, Russia}

\noindent
\emph{Email address: skresan@math.nsc.ru}

\end{document}